\documentclass[10pt]{article}         

\usepackage{mathrsfs}
\usepackage{amsmath}
\usepackage{amsfonts}
\usepackage{amssymb}
\usepackage{amsthm}
\usepackage{caption}
\usepackage{subfigure}
\usepackage{cite}
\usepackage{color}
\usepackage{graphicx}
\usepackage{subfigure}
\usepackage{float}
\usepackage{paralist}
\usepackage[colorlinks,linkcolor=blue,anchorcolor=blue,citecolor=blue,  urlcolor=blue]{hyperref}
\usepackage{indentfirst}
\usepackage{authblk}
\usepackage[T1]{fontenc}
\usepackage{amscd}
\usepackage{latexsym}
\usepackage{verbatim}
\usepackage{enumerate}
\usepackage{fullpage}
\usepackage{xcolor}
\usepackage{titlesec}
\usepackage{titletoc}
\usepackage{nameref}
\usepackage{appendix}
\usepackage{doi}
\newcommand{\red}{\textcolor{red}}
\newtheorem{theorem}{Theorem}[section]
\newtheorem{lemma}{Lemma}[section]
\newtheorem{definition}{Definition}[section]

\newtheorem{remark}{Remark}[section]

\numberwithin{equation}{section}
\newcommand{\RNum}[1]{\uppercase\expandafter{\romannumeral #1\relax}}
\title{Sharp power concavity of two relevant  free boundary problems of reaction-diffusion type}
\author{Qingyou He}
\date{}           

\begin{document}
\maketitle
\begin{abstract} The porous medium type reaction-diffusion equation and the Hele-Shaw problem are two free boundary problems linked through the incompressible (Hele-Shaw) limit.  We investigate and compare the sharp power concavities of the pressures on their respective supports for the two free boundary problems. For the pressure of the porous medium type reaction-diffusion equation, the $\frac{1}{2}$-concavity preserves all the time, while $\alpha$-concavity for $\alpha\in[0,\frac{1}{2})\cup(\frac{1}{2},1]$ does not persist in time. In contrast, in the case of the pressure for the Hele-Shaw problem, $\alpha$-concavity with $\alpha\in[0,\frac{1}{2}]$ is maintained all the while and  $\frac{1}{2}$ acts as the largest index. The intuitive explanation for the difference between the two free boundary problems is that, although the Hele-Shaw problem is the incompressible limit of the porous medium-type reaction-diffusion equation, it is no longer a degenerate parabolic equation. Furthermore, for the pressure of the porous medium type reaction-diffusion equation, the non-degenerate estimate is established by means of the derived concave properties, indicating that the spatial Lipschitz regularity in the whole space is sharp.
\end{abstract}

\noindent 2020 {\bf MSC.} 35K57; 35E10; 35R35; 76D27.

\section{Introduction}
\paragraph{Models and setting.}
The nonlinear reaction-diffusion equation serves as a fundamental model for describing a wide range of complex processes across disciplines such as Biology, Ecology, Physics, and Chemistry. For detailed discussions, readers can refer to sources such as \cite{BD2009, CD2017B, FKA1937, PERTHAME2014, WTP1997}. A prominent example is its application in modeling tumor growth as discussed in\cite{BD2009,PERTHAME2014}, where the porous medium type reaction-diffusion equation forms a free boundary problem when the initial data is compactly supported and captures the dynamics of cell proliferation. In this context, cell division is regulated by biomechanical contact inhibition, and ceases once the cell density reaches a critical threshold. Furthermore, the incompressible limit of this porous medium type reaction-diffusion equation leads to a Hele-Shaw type free boundary problem~\cite{PERTHAME2014,DAVID2021,DN2023}, which serves to simulate tumor propagation in a manner consistent with more complex models presented in \cite{BCGS2010, LFJC2010}. In the case of the Hele-Shaw problem, the density is suppressed by the stiff pressure to remain no more than $1$. 

\vspace{2mm}

The porous medium type reaction-diffusion equation reads as 
\begin{equation}\label{rde}
   \partial_t u=\Delta u^m+uG(P),\quad (x,t)\in\mathbb{R}^N\times \mathbb{R}_+
\end{equation}
with the initial data
\[u_0(x)=u(x,0)\geq0,\quad x\in\mathbb{R}^N,\]
where $u$ is the density,  $G(\cdot)$ is a given decreasing function with $P_M>0$ such that $G(P_M)=0$, $m>1$ (slow diffusion) denotes the diffusion exponent, the pressure \begin{equation}\label{pressure}
P:=\frac{m}{m-1}u^{m-1}\quad\text{with } P_0:=\frac{m}{m-1}u_0^{m-1} 
\end{equation} solves a degenerate parabolic equation 
\begin{equation}\label{pe}
\begin{aligned}
\partial_t P=rP\Delta P+|\nabla P|^2+rPG(P),\quad r:=m-1.
\end{aligned}
\end{equation}
Two classical examples are  the tumor growth model with $G(P)=P_M-P$ from~\cite{PERTHAME2014} and the Fisher-KPP equation with $G(P)=g(u)=u_M-u$ from~\cite{DQZ2020}.
 The main feature of the porous medium type equation~\eqref{rde} is that the solution $u$ remains compactly supported all the time if the initial density $u_0$ owns compact support, and the free boundary is referred to the boundary of the support for the pressure $P$.
\vspace{2mm}

The Hele-Shaw problem corresponding to \eqref{rde}-\eqref{pe} is expressed as
\begin{equation}\label{hs}
\begin{cases}
\partial_t u=\Delta P+u G(P),&\\
0\leq u\leq 1,\quad (1-u)P=0,&\\
P(\Delta P+G(P))=0,\quad (x,t)\in\mathbb{R}^N\times \mathbb{R}_+,&
\end{cases}
\end{equation}
with the  initial data
\[u_0(x)=u(x,0)\geq0,\quad x\in\mathbb{R}^N,\]
where $u$ and $P$ represent the density and the stiff pressure respectively. $\eqref{hs}_2$ is called as  the Hele-Shaw graph, which links the density and pressure differently from \eqref{pressure} and is the main feature of the Hele-Shaw problem. The complementarity relationship $\eqref{hs}_3$ is a degenerate elliptic equation to describe the stiff pressure $P$. Here, the given  funtion $G$ with $G(0)>0$ is a decreasing function.  The free boundaries for the Hele-Shaw problem \eqref{hs} refer to the support boundaries  for the pressure and the density when the initial density $u_0$ is compactly supported. What needs to be emphasized is that we study the support boundary of the pressure in this paper.

\vspace{2mm}

Indeed, $\eqref{hs}_1$, $\eqref{hs}_2$, and $\eqref{hs}_3$ for the Hele-Shaw problem can be formally derived from \eqref{rde}, \eqref{pressure}, and \eqref{pe} respectively as the diffusion exponent $m$ tends to infinity, this process is commonly referred to as the incompressible (or Hele-Shaw) limit. Over the past decades, rigorous analytical proofs of the incompressible limit for the porous medium type reaction-diffusion equation \eqref{rde} have been established in \cite{PERTHAME2014,DN2023,KT2018,GKM2022,KM2023,DS2024,JY2024}. The regularity of the solution and its support  boundary for both two free boundary problems \eqref{rde} and \eqref{hs} are well studied in \cite{HQ2024,MELLET2017} respectively. 

\vspace{2mm}

Unlike the case of the porous medium type equation \eqref{rde}, the pressure in the Hele-Shaw problem $\eqref{hs}$ is no longer governed by a degenerate parabolic equation like \eqref{pe}. If the density and the pressure share the same support, and the boundary of this support is differentiable to the second order, the two free boundary problems \eqref{rde} and $\eqref{hs}$ exhibit  the same form of the normal velocity at the boundary, a characteristic commonly referred to as Darcy’s law:
\[V(t)=\nabla P\cdot\vec{n},\]
where $\vec{n}$ is the outward normal to the support boundary 
$\partial\{P>0\}$. In addition, due to the Hele-Shaw relation $\eqref{hs}_2$, for 
$G(\cdot)=g(u)$ in $\eqref{hs}_1$ of the Fisher-KPP equation, it holds
\[G(P)\equiv  const\equiv g(1)=G(0) \text{ on }\{P>0\}.\]

\vspace{2mm}

 To precisely describe the concavity of functions, the power concavity of the non-negative function $f:\mathbb{R}^N\to[0,\infty)$ is defined by
\begin{definition}[Power concavity]For $\alpha\in[0,1]$ and a non-negative function $f$, it is said that $f$ is $\alpha$-concave if $f^\alpha$ is concave on its support, and that $f$ is $0$-concave if $\log f$ is concave on its support. In particular, if $f$ is twice differentiable, the concavity is equivalent to f having a negative semi-definite Hessian matrix.
\end{definition}

\vspace{2mm}

The purpose of this article is to mainly investigate and compare the sharp power concavities of the pressure in both the porous medium type reaction-diffusion equation \eqref{rde} and the corresponding Hele-Shaw problem \eqref{hs}. To be exact, we aim to find out the sharp index of the power concavity for the pressure of the two free boundary problems. Although the Hele-Shaw problem \eqref{hs} arises as the incompressible limit of the reaction-diffusion equation of porous medium type \eqref{rde}, the interest is that our analysis from the angle of sharp power concavity seeks to uncover notable differences and highlight key distinctions between the two connected free boundary problems \eqref{rde} and \eqref{hs}. In addition, based on the verified concave properties, we are committed to getting the upper and lower bounds of the pressure gradient for the porous medium type reaction-diffusion equation~\eqref{rde}. This, in turn, demonstrates that the spatial Lipschitz continuity of the pressure in the whole space is sharp.

\vspace{2mm}

\paragraph{Related studies.} Over the past decades, the convexity of the solution and its level set for free boundary problems with degenerate diffusion such as \eqref{rde} have long been a significant topic, particularly related to the geometric properties of the solution. 
Let us look back some related studies on the power concavity for some classical free boundary problems. The classical degenerate parabolic equation is the porous medium equation (PME)
\begin{equation}\label{pme}\partial_t u=\Delta u^m.\end{equation}
For the pressure of the high-dimensional PME \eqref{pme}, the preservation of $\frac{1}{2}$-concavity all the time is verified in \cite{DHL2001}, and non-preservation of $\alpha$-concavity for $\alpha\in[0,1]\backslash\{\frac{1}{2}\} $  holds in time; see \cite{CW2024,ST2024}. In addition, authors in \cite{DHL2001,DH1998} proves that the pressure for the PME \eqref{pme} is smooth to the support boundary. Before this, authors in \cite{IS2010} demonstrate that there exists some $0<\alpha<\frac{1}{2}$ such that $\frac{1}{2}$-concavity is not generally preserved in time.  However, for the one-dimensional PME \eqref{pme}, the authors of  \cite{BV1987} already shown that the concavity of pressure is preserved all the time. In this paper, we partially hope to extend these results on the power concavity of the high-dimensional PME \eqref{pme} to the case of the PME with a source term~\eqref{rde}. Provided that the initial density is compactly supported, even if the initial support is not initially convex, authors in \cite{LV2003} proved that the support 
\[\{x\in\mathbb{R}^N: P(x,t)>0\}\] 
for the PME \eqref{pme} becomes convex and the pressure 
evolves into a concave function on its spatial support for sufficiently large time.  For the solution of the one-phase Stefan problem, which satisfies the heat equation in the interior and where the free boundary moves in the outward normal direction at a speed equal to the magnitude of the pressure gradient, $\alpha$-concavity for $\alpha\in[0,\frac{1}{2})$ is not preserved over time; cf.~\cite{CW2021}. For the Hele-Shaw problem \eqref{hs} without reaction term, the work \cite{DL2004} shows that the preservation of log-concavity holds all the while. Moreover, the preservation of log-concavity all through is also valid for the flame propagation problem studied in~\cite{DL2002}.

\vspace{2mm}

In addition, since the pressure $P$ of the Hele-Shaw problem $\eqref{hs}_3$ solves an elliptic equation \eqref{hss} with a Dirichlet boundary condition at any fixed time,  it is necessary to recall some important works on the power concavity of the solution for the elliptic problems.  Considering the simple  and classical elliptic boundary value problem:
\begin{equation}\label{laplace1}
\begin{cases}
    \Delta u=-1,\quad&\text{in }\Omega,\\
    u=0,\quad&\text{on }\partial\Omega
\end{cases}\end{equation}
in the two-dimensional bounded convex domain $\Omega$ with Lipschitz boundary, authors in \cite{ML1971} show that $\sqrt{u}$ is strict convex in $\Omega$. Then, this conclusion of  \cite{ML1971} was generalized  to the higher-dimensional cases of \eqref{laplace1} in \cite{KB1986,KA1985} via extending the maximum principle given in \cite{KN1983}. In particular, authors in \cite{KA1985} pointed out that the $\frac{1}{2}$-concavity  is sharp in $N$-dimensional $(N\geq2)$ elliptic equation~\eqref{laplace1}. For the solution of the generalized elliptic equation, authors in \cite{BS2013} proved a sufficient condition of the power concavity via the method of viscosity solution.

\paragraph{Main results.}

The index $\frac{1}{2}$ for time preservation of power concavity  is  sharp not only for the porous medium  equation with a source term  \eqref{rde} but also for the corresponding Hele-Shaw problem \eqref{hs}. The key difference lies in the role of this index: for the porous medium type reaction-diffusion equation \eqref{rde}, $\frac{1}{2}$ is the unique index that ensures the preservation of concavity all the time. In contrast, for the Hele-Shaw problem \eqref{hs}, $\frac{1}{2}$ serves as an upper bound of the index for power concave preservation all the while.

To achieve the goals, it is assumed that the initial pressure $P_0$ satisfies 
\begin{equation}\label{ini}
\begin{aligned}
&\ \ \ \ \ \ \  \ \ \ \  \ \ \ \ \ \ \ \ 0\leq P_0\leq P_H\quad\text{for some }P_H\geq P_M>0,\\
P_0& \text{ is compactly supported and smooth to its support boundary, and }P_0 \text{ is }\frac{1}{2}\text{-concave},
\end{aligned}
\end{equation}
\begin{equation}\label{ini2}
\Delta P_0+G(P_0)\geq -K\quad\text{ for some }K>0,
\end{equation}
\begin{equation}\label{ini3}
P_0+|\nabla P_0|^2\geq c>0\quad\text{on }\{P_0>0\}\text{ and }|\nabla P_0|\leq C \text{ for some }C,\ c>0.
\end{equation}
Let $G$ be supposed to satisfy
\begin{equation}\label{concavityc}
\sup\limits_{P\in[0,P_H]}3G'(P)+2PG''(P)\leq 0,
\end{equation}
\begin{equation}\label{g0b}
|G'(P)|+|PG''(P)|\leq AP^{\gamma-1}\quad \text{for }P\in(0,P_H] \text{ and some }A, \gamma>0,
\end{equation}
\begin{equation}\label{abc}
\inf\limits_{P\in[0,P_H]}G(P)-PG'(P)\geq 0,
\end{equation}
\begin{equation}\label{abc1}
\sup\limits_{0\leq P\leq P_H}-PG'(P)\leq L-K\text{ for some }L>K>0.
\end{equation}
\begin{remark}We verify two classical examples, like both the tumor growth models \cite{PERTHAME2014} and the Fisher-KPP equation \cite{DQZ2020}, to show the generality of the conditions \eqref{concavityc}-\eqref{abc1} on $G$. For the case of  tumor growth model $G(P)\in\mathcal{C}^2([0,P_H])$  with $G'(P)\leq 0$ and $G''(P)\leq 0$ (see \cite{PERTHAME2014}), the conditions \eqref{concavityc}-\eqref{g0b} evidently hold for all $m>1$. As for the Fihser-KPP equation with the simple reaction term  $G(P)=u_M-u=u_M-(\frac{m-1}{m})^{\frac{1}{m-1}}P^{\frac{1}{m-1}}$,  the conditions \eqref{concavityc}-\eqref{g0b} for all $m>1$ are valid by calculating
\[\begin{aligned}3G'(P)&+2PG''(P)\\
=&-3(\frac{m-1}{m})^{\frac{1}{m-1}}\frac{1}{m-1}P^{\frac{1}{m-1}-1}-2(\frac{m-1}{m})^{\frac{1}{m-1}}\frac{1}{m-1}(\frac{1}{m-1}-1)P^{\frac{1}{m-1}-1}\\
=&-(\frac{m-1}{m})^{\frac{1}{m-1}}\frac{1}{m-1}(1+\frac{2}{m-1})P^{\frac{1}{m-1}-1}\leq 0,
\end{aligned}\]
and 
\[ |G'(P)|+|PG''(P)|\leq \max\{(\frac{m-1}{m})^{\frac{1}{m-1}}\frac{1}{m-1},(\frac{m-1}{m})^{\frac{1}{m-1}}\frac{1}{m-1}|\frac{1}{m-1}-1|\}P^{\frac{1}{m-1}-1}.\]
In addition, it is easy to check that the conditions \eqref{abc}-\eqref{abc1} hold for both the tumor growth models \cite{PERTHAME2014} and the Fisher-KPP equation \cite{DQZ2020}.

\end{remark}

\begin{theorem}[Sharp concavity for reaction-diffusion equation]\label{concavityrde}Under the conditions~\eqref{concavityc} and \eqref{g0b} on $G$, assume that the initial pressure $P_0$ satisfies \eqref{ini} and the pressure $P$ for the porous medium type reaction-diffusion equation \eqref{rde} with  $r:=m-1>0$ is smooth to the boundary of its support. Then, the pressure $P$  remains $\frac{1}{2}$-concave all the time. Moreover, for any $\alpha\in [0,1]\backslash\{\frac{1}{2}\}$, there exists at least one $\alpha$-concave radially initial  pressure $P_0$ such that the $\alpha$-concavity of the pressure $P$ loses instantaneously.
\end{theorem}

Based on the $\frac{1}{2}$-concavity of pressure for the porous medium type reaction-diffusion equation~\eqref{rde} gotten in Theorem \ref{concavityrde}, we can obtain the spatially non-degenerate property of this pressure. 
\begin{theorem}[Sharp Lipschitz continuity]\label{slc} 
Under the initial assumptions \eqref{ini}-\eqref{ini3} and the conditions \eqref{concavityc}, \eqref{abc} and \eqref{abc1} on $G$, and assume that the  pressure $P$ for the porous medium type reaction-diffusion equation \eqref{rde} is smooth to its support boundary. Then, for all $r:=m-1>0$, the spatially non-degenerate estimate of the pressure $P$ holds as
\[|\nabla P|\leq Ce^{rG(0)t},\ t\geq 0,\]
and
\[\big[1+(t+\frac{2}{Kr+2})rG(P)\big]P+(t+\frac{2}{Kr+2})(1+\frac{r}{2})|\nabla P|^2\geq \frac{c}{Kr+2}e^{-rLt}\quad \text{on }\{P>0\}.\]
\end{theorem}

\begin{remark}
The above theorem further indicates that the pressure linearly degenerates to its spital support boundary at any fixed time, reinforcing the conclusion that the spatial Lipschitz continuity is indeed the sharp regularity as studied in \cite{HQ2024}.
\end{remark}

For the Hele-Shaw problem, we give the assumptions of $G$:
\begin{equation}\label{hsc}
G(\cdot)\in\mathcal{C}^2([0,\infty)),\quad G(0)>0,\quad G',\ G''\leq 0, 
\end{equation}
where $G(\cdot)\equiv\text{const}$ on $\{P>0\}$ is corresponding to the case of incompressible limit for the Fisher-KPP equation $G(P):=u_M-u$ in \eqref{rde}. 
At $t=0$, the initial pressure $P_0$ solves an elliptic equation with Dirichlet boundary as
\begin{equation}\label{hs0}
\begin{cases}
-\Delta P_0=G(P_0),\quad &\text{in }\Omega_0,\\
P_0=0,\quad&\text{on }\partial\Omega_0,
\end{cases}
\end{equation}
where 
\[\Omega_0:=\{u_0=1\}.\]

\begin{theorem}[Sharp concavity for Hele-Shaw problem]\label{schsp} Under the conditions \eqref{hsc} on $G$, assume that the initial density $0\leq u_0\leq 1$ is compactly supported and the level set $\Omega_0$  is convex with smooth boundary. Then, the initial pressure $P_0$, solving the elliptic problem \eqref{hs0}, is $\alpha$-concave for any  $\alpha\in[0,\frac{1}{2}]$, and $\frac{1}{2}$ serves as the largest index for the power concavity. Let $(u,P)$ be the solution for the Hele-Shaw problem \eqref{hs}.  Assume further that the stiff pressure $P$ is smooth to its support boundary, then the preservation of spatial $\alpha$-concavity for any $\alpha\in[0,\frac{1}{2}]$ holds all the time
\end{theorem}
\begin{remark}
    For the porous medium type reaction-diffusion  equation \eqref{rde}, $\frac{1}{2}$ is the only exponent with respect to the preservation of the concavity for the pressure all the time. However, turning to the Hele-Shaw problem~\eqref{hs}, $\frac{1}{2}$ is the upper bound of exponent concerned to the preservation of concavity over time. The reason is that the pressure for the porous medium equation solves a degenerate parabolic equation \eqref{pe}, but the pressure for the Hele-Shaw problem \eqref{hs} solves an elliptic equation \eqref{hss} in each fixed time and the free boundary can form a parabolic equation with fractional diffusion; see \cite{KT2021}. Hence, the shape of the support determines the shape of the pressure.

\end{remark}
\paragraph{Organisation of this paper.} In Section \ref{rdpme}, we prove the sharp power concavity of the pressure for the porous medium type reaction-diffusion equation (Theorem \ref{concavityrde}) and the further spatially non-degenerate estimate of the pressure for this equation (Theorem \ref{slc}). In Section \ref{hsp}, we verify  the sharp power concavity of the pressure for the Hele-Shaw problem (Theorem \ref{schsp}). Finally, conclusions and perspectives will be given in Section \ref{cp}.

\vspace{2mm}

To simplify the notation, we denote $\partial_{x_k}P$ by $P_k$ in the rest of this paper.
\section{Porous medium type reaction-diffusion equation}\label{rdpme}
We show that the $\frac{1}{2}$-concavity of pressure remains all the time for the porous medium type reaction-diffusion equation~\eqref{rde}. Then, by constructing the counterexamples, we prove the non-preservation of $\alpha$-concavity of the pressure for any $\alpha\in [0,1]/\{\frac{1}{2}\}$ over time. Furthermore, based on the gotten $\frac{1}{2}$-concavity of pressure,  the spatially non-degenerate estimate can be established.

\subsection{Root concavity}
 Authors in \cite{DHL2001} proved that the $\frac{1}{2}$-concavity of pressure remains all the time for the PME \eqref{pme}. Via following the approach of \cite{DHL2001}, then, overcoming the logistic growth effect, we are going to prove that concave  preservation for the pressure  is valid all the time. 
\begin{lemma}\label{12concave}
For the porous medium type reaction-diffusion equation \eqref{rde}, assume that the initial pressure $P_0$ is $\frac{1}{2}$-concave  and the pressure $P$ is smooth to its support boundary. Then, $P$ remains spatially $\frac{1}{2}$-concave all the time. 
\end{lemma}
\begin{proof}[\bf Proof] If 
\begin{equation}\label{aijn}A_{ij}=P_{ij}-\frac{P_iP_j}{2P}\leq 0\end{equation}
holds in the sense of symmetric matrix, then 
\begin{equation*}
D_{ij}^2(\sqrt{P})=\frac{1}{2\sqrt{P}}(P_{ij}-\frac{P_iP_j}{2P})
\end{equation*}
is a non-positive symmetric matrix which means that the pressure $P$ is $\frac{1}{2}$-concave. 
To verify \eqref{aijn},  we will show 
\begin{equation}\label{aijdelta}A_{ij}\leq \textbf{0}\text{ in the sense of symmetric matrix}.\end{equation}
For any fixed time $T>0$, we next prove that the quadratic satisfies
\begin{equation}\label{qf}
Z=(P_{ij}-\frac{P_iP_j}{2P}-\Psi I_{ij})V^iV^j<0,
\end{equation}
for any non zero vector filed $V(x,t)\in\mathbb{R}^N$ in $[0, T]$, where the chosen positive function $\Psi=\Psi(t)$ solves
\begin{equation}\label{psi}\Psi'\geq 2C(\Psi+\Psi^2)\end{equation}
\textcolor{red}{with} \begin{equation}\label{Cdef}
C\geq r+2+r\big(c+G(0)\big)+C_1
\end{equation}with $c>\max\{0, \max\limits_{P\in[0,P_H]}-(G(P)+PG'(P))\}$ and $C_1=\sup\limits_{\overline{\{P>0\}}\cap\{ t\leq T\}}|\Delta P|$.

The expected estimate \eqref{aijdelta} is verified by contradiction. If \eqref{qf} does not hold, for some  $\Psi(0)>0$ and $T>0$, there exists $x_0\in\overline{\Omega(t_0)}:=\overline{\{P(\cdot,t_0)>0\}}$, $V_0$ with the smallest $t_0\in(0,T]$ such that 
\[Z(x_0,t_0)=0.\]
  
The rest of proof is divided into two cases including the interior case and the boundary case.
\paragraph{Case 1, estimate in the interior.} Assume that $x_0$ is the interior point at the time $t_0$.
We extend $V_0$ to be a smooth vector field $V$ in a neighborhood of $(x_0,t_0)$  satisfying
\begin{equation*}V^i_j=\frac{1}{2P}(P_kV^k)I_{ij}\text{ and }V^j_t=0,\ V^j_{kk}=0\text{ at } (x_0,t_0).\end{equation*}
On the one hand, differentiating \eqref{pe}, we have 
\[P_{it}=rPP_{ikk}+rP_iP_{kk}+2P_kP_{ik}+rP_i(G(P)+PG'(P))\]
and 
\begin{equation*}
\begin{aligned}P_{ijt}=&rPP_{ijkk}+2rP_iP_{jkk}+rP_{ij}P_{kk}+2P_{ik}P_{jk}+2P_{k}P_{ijk}\\
&+rP_{ij}(G(P)+PG'(P))+rP_iP_j(2G'(P)+PG''(P)).
\end{aligned}
\end{equation*}
Then, it holds at $(x_0,t_0)$ that 
\begin{equation}\label{kk1}
\begin{aligned}
(P_{ij}V^iV^j)_t=\big[&rPP_{ijkk}+2rP_iP_{jkk}+rP_{ij}P_{kk}+2P_{k}P_{ijk}+2P_{ik}P_{jk}\\
&+rP_{ij}(G(P)+PG'(P))+rP_iP_j(2G'(P)+PG''(P))\big]V^iV^j.  \end{aligned}
\end{equation}
In addition, we calculate directly and obtain 
\[(P_{ij}V^iV^j)_k=[P_{ijk}+\frac{P_{jk}P_i}{P}]V^iV^j\]
and 
\begin{equation*}\label{pijkk}
    P(P_{ij}V^iV^j)_{kk}=[PP_{ijkk}+2P_{jkk}P_i+\frac{P_{kk}P_iP_j}{2P}]V^iV^j\quad\text{at }(x_0,t_0).
\end{equation*}
 Hence, it holds 
\begin{equation*}
    \begin{aligned}
(P_{ij}V^iV^j)_t=& rP(P_{ij}V^iV^j)_{kk}+2P_k(P_{ij}V^iV^j)_k\\
&+rP_{kk}[(P_{ij}-\frac{P_iP_j}{2P})V^iV^j]-2\frac{P_{jk}P_iP_k}{P}V^iV^j\\
&+2P_{ik}P_{jk}V^iV^j+[rP_{ij}\big(G(P)+PG'(P)\big)\\
&+rP_iP_j\big(2G'(P)+PG''(P)\big)]V^iV^j\text{ at }(x_0,t_0). 
    \end{aligned}
\end{equation*}
 On the other hand, we have 
\begin{equation*}
\begin{aligned}
(\frac{P_iP_j}{P}V^iV^j)_t=&[2rP_jP_{ikk}+\frac{rP_iP_jP_{kk}}{P}+\frac{4P_jP_kP_{ik}}{P}-\frac{P_iP_jP_k^2}{P^2}]V^iV^j\\
&+\big[\frac{2P_iP_j(G(P)+PG'(P))}{P}-\frac{P_iP_jG(P)}{P}\big]V^iV^j,
\end{aligned}
\end{equation*}
\begin{equation*}
(\frac{P_iP_j}{P}V^iV^j)_k=\frac{2P_{ik}P_j}{P}V^iV^j,
\end{equation*}
and 
\begin{equation*}\label{pipjp}
    P(\frac{P_iP_j}{P}V^iV^j)_{kk}=\big[2P_{ikk}P_j+\frac{P_{kk}P_iP_j}{P}+2\big(P_{ik}-\frac{P_iP_k}{2P}\big)\big(P_{jk}-\frac{P_jP_k}{2P}\big)\big]V^iV^j
\end{equation*} 
at $(x_0,t_0)$.
It follows
\begin{equation}\label{k2}
\begin{aligned}
    (\frac{P_iP_j}{2P}V^iV^j)_t=&rP(\frac{P_iP_j}{2P}V^iV^j)_{kk}+2P_k(\frac{P_iP_j}{2P}V^iV^j)_k\\
    &-r\big(P_{ik}-\frac{P_iP_k}{2P}\big)\big(P_{jk}-\frac{P_jP_k}{2P}\big)V^iV^j-\frac{P_iP_jP_k^2}{2P^2}V^iV^j\\
    &+rP_iP_j[\frac{G(P)+PG'(P)}{2P}+\frac{G'(P)}{2}]V^iV^j.\\
\end{aligned}
\end{equation}

The vector $V_0$ is a null eigenvector for the symmetric  matrix
\[P_{ij}-\frac{P_iP_j}{2P}-\Psi I_{ij}\quad \text{at }(x_0,t_0),\]
which implies
\[(P_{ij}-\frac{P_iP_j}{2P})V^i=\Psi V^j\quad\text{ at }(x_0,t_0).\]
Hence, combining the above \eqref{kk1} and \eqref{k2} yields
\begin{equation*}
Z_t\leq P Z_{kk}+2rP_kZ_k+[r\Psi^2-\Psi_t]|V|^2+R
+\frac{r(P_iV_i)^2}{2}\big(3G'(P)+2PG''(P)\big)
\end{equation*}
with 
\begin{equation*}
    \begin{aligned}
  R=&2P_{ik}P_{jk}V^iV^j-2\frac{P_{jk}P_iP_k}{P}V^iV^j+\frac{P_{i}P_jP_k^2}{2P^2}V^iV^j+rP_{ij}(G(P)+PG'(P))\\
  &-r(\frac{G(P)+PG'(P)}{2P})P_iP_jV^iV^j\\
  =&2\big(P_{ik}-\frac{P_iP_k}{2P}\big)\big(P_{jk}-\frac{P_jP_k}{2P}\big)V^iV^j+r(G(P)+PG'(P))(P_{ij}-\frac{P_iP_j}{2P})V^iV^j\\
  \leq & (cr\Psi+2\Psi^2)|V|^2,
    \end{aligned}
\end{equation*}
where $c>\max\{0, \sup\limits_{P\in[0,P_H]}-(G(P)+PG'(P))\}$.
It further follows from the condition \eqref{concavityc} on $G$ and the inequality \eqref{psi} on $\Psi$ that 
\begin{equation}\label{zt}
\begin{aligned}
Z_t\leq&  PZ_{kk}+rP_kZ_k+\big[(r+2)\Psi^2+cr\Psi-\Psi'\big]|V|^2\\
\leq&- C(\Psi+\Psi^2)|V|^2\quad\text{at }(x_0,t_0),
\end{aligned}
\end{equation}
where $C>0$ is given by \eqref{Cdef}. Since $Z$ obtain the maximum at the interior point $(x_0,t_0)$ for $t\leq t_0$, it is immediate that
\[Z_t\geq 0,\quad Z_{kk}\leq 0,\quad Z_k=0\text{ at }(x_0,t_0).\]
Inserting the above into \eqref{zt} derives a contradiction
\[0\leq -C(\Psi+\Psi^2)|V|^2< 0,\]
which means that at the first time $t_0$, $Z(x_0,t_0)=0$ does not occur in the interior.

\paragraph{Case 2, estimate on the boundary.} Suppose that $x_0\in\partial\{x\in\mathbb{R}^N:P(x,t_0)>0\}$ occurs at the spatial support boundary, that is,
\begin{equation}\label{zb}Z(x_0,t_0)=0\end{equation}
for some nonlinear vector $|V_0|=1$, and $Z\leq 0$ for all $0\leq t\leq t_0$. 

Since $P=0$ on the boundary and $P$ is smooth to its support boundary, it concludes that on the boundary, $V$ is tangent to the boundary
\begin{equation}\label{pivi0}P_i(x,t)V^i=0,\quad 0\leq t\leq t_0.\end{equation}
If not,  that is $|P_iV^i|>0$  on the boundary, the definition of $Z$ \eqref{qf}  implies 
\[Z=-\infty,\]
that is impossible because of the continuity of $Z$. 
Since 
$P$ has the degeneracy of one order to the boundary, the estimate \eqref{pivi0} derives 
\begin{equation}\label{pvijp}
    \frac{P_iV^iP_jV^j}{P}=0
\end{equation}
on the boundary.  Furthermore,  we pick a path $x(t)$ for $t\leq t_0$ with $x(t)\in\partial \Omega(t)$ and $x(t_0)=x_0$. We also choose a path $V(t)$ for $t\leq t_0$ with $V(t)$ always tangent to the boundary at the point $x(t)$  for $t\leq t_0$, and $V(t_0)=V_0$. 

For the above $x_0,t_0,V_0$, \cite[Lemma 2.2]{DHL2001} provides a preliminary estimate 
\begin{equation}\label{np}
P_kP_{ijk}V_0^iV_0^j\leq 0\text{ at }x_0\in\partial\Omega(t_0).
\end{equation}

The chosen $x(t)$ satisfies $P(x(t),t)=0$, then the yielded
$\frac{d}{dt} P(x(t),t)=0$
 makes
\[P_t+P_k\frac{dx^k}{dt}=0.\]
Due to the pressure equation~\eqref{pe}, we obtain
\[P_t=rPP_{kk}+|P_k|^2+rPG(P)=|P_k|^2\ \text{ on the support  boundary.}\] 
 Thus,  \[P_k(\frac{dx^k}{dt}+P_k)=0\] holds by choosing 
\[\frac{dx^k}{dt}=-P_k.\]
Thanks to $P_kV^k=0$ along $x(t)\in\partial \Omega(t)$ from \eqref{pivi0}, it is satisfied 
\[\frac{d}{dt}(P_kV^k)=0.\]
The chain rule of compound derivative yields 
\begin{equation}\label{pkvkd}P_k\frac{d}{dt}V^k+P_{kt}V^k+P_{jk}\frac{dx^j}{dt}V^k=0.\end{equation}
Using the pressure equation \eqref{pe}, it follows 
\[P_{it}=rPP_{ikk}+rP_iP_{kk}+2P_{ik}P_k+rP_i(G(P)+PG'(P)),\]
which further implies 
\begin{equation}\label{pit}P_{ti}V^i=2P_kP_{ik}V^i\end{equation}
 on the boundary. Substituting \eqref{pit} into \eqref{pkvkd} yields 
\[P_k(\frac{dV^k}{dt}+P_{ik}V^i)=0.\]
It is immediate to hold by choosing
\[\frac{dV^k}{dt}=-P_{ik}V^i.\]

On account of \eqref{pvijp} along the path $x(t)$, we naturally consider the function 
\[Q(t)=Z(x(t),t)=(P_{ij}-\Psi I_{ij})V^iV^j,\]
with $Q(t_0)=0$ and $Q(t)\leq 0\text{ for } t\leq t_0$, which implies 
\begin{equation}\label{aqt}
\frac{dQ}{dt}\geq0\quad \text{at }t=t_0.
\end{equation}

We directly compute the time derivative 
\[\frac{dQ}{dt}=(P_{ijt}-\Psi' I_{ij})V^iV^j+P_{ijk}\frac{dx^k}{dt}V^iV^j+2(P_{ij}-\Psi I_{ij})\frac{dV^i}{dt}V^j.\]
In addition, it follows 
\begin{equation*}
\begin{aligned}P_{ijt}=&rPP_{ijkk}+rP_iP_{jkk}+rP_jP_{ikk}+rP_{ij}P_{kk}+2P_{ik}P_{jk}+2P_{k}P_{ijk}\\
&+rP_{ij}(G(P)+PG'(P))+rP_iP_j(2G'(P)+PG''(P)).
\end{aligned}
\end{equation*}
Due to $P=0$ and $P_iV^i=0$ on the path $(x(t),t)$, we get 
\[P_{ijt}V^iV^j=2P_{k}P_{ijk}V^iV^j+rP_{kk}(P_{ij}V^iV^j)+2(P_{ik}V^i)(P_{jk}V^j)+rP_{ij}V^iV^jG(0).\]
Taking  $P_{ij}V^i=\Psi V^j$ at $t=t_0$, $\frac{dx^k}{dt}=-rP_k$, $\frac{dV^k}{dt}=-rP_{ik}V^i$ and the estimate \eqref{np} into account yields 
\begin{equation}\label{qt}
\frac{dQ}{dt}\leq \big(-\Psi'+P_{kk}\Psi+2\Psi^2+G(0)\Psi \big)|V|^2\leq \big(-\Psi'+C\Psi+C\Psi^2 \big)|V|^2 \quad\text{at }x_0\in\partial\Omega(t_0).
\end{equation}
Since  $\Psi'>2C(\Psi+\Psi^2)$ with $C$ being defined by \eqref{Cdef}, the inequality \eqref{qt} yields
\[\frac{dQ}{dt}<-C(\Psi+\Psi^2)|V|^2<0\quad\text{at }t=t_0,\] 
which is in contradiction with \eqref{aqt}.

\vspace{2mm}

Combining the discussions of both  \textbf{Case 1, estimate in the interior} and \textbf{Case 2, estimate on the boundary}, it immediately becomes clear that $Z<0$ for all $x\in \overline{\Omega(t)}\cap [0,T]$ for any $T>0$. Finally, since $\Psi$ can be as small as we expect, $\sqrt{P}$ is spatially concave on its support. 
\end{proof}
\subsection{Non preservation of power concavity}
We follow the approach of \cite{ST2024,CW2024} to prove the non-preservation of $\alpha$-concavity with $\alpha\in[0,1]\backslash\{\frac{1}{2}\}$ over time, which means that the index~$\frac{1}{2}$ for the preservation of concavity in time is unique and sharp for the porous medium type reaction-diffusion equation~\eqref{rde}.  To be more precise, the constructed radially symmetric concave initial datum in \cite{ST2024,CW2024} means that the nonpositivity of the first eigenvalue of the Hessian matrix for the $\alpha$th-power pressure loses instantaneously. Indeed, compared with previous works \cite{ST2024,CW2024} on the  PME \eqref{pme}, the source term causes some new difficulties including the absence of scaling invariance. 

Let us start with stating the main result:
\begin{theorem}\label{npc}
Let $B$ be the open unit ball in $\mathbb{R}^N$ centered at the origin with $N\geq 2$. For any given $\alpha\in[0,1]\backslash\{\frac{1}{2}\}$, there exists $v_0\in C^\infty(\overline{B})$ which is strictly positive on $B$ and vanished on $\partial B$ with the following properties:
\begin{itemize}
    \item[(1)] $v_0$ is $\alpha$-concave on $B$.
    \item[(2)] $\nabla v_0$ does not vanish at any point of $\partial B$.
    \item[(3)] Let $v(t)$ be the solution to \eqref{rde} starting with $v_0$. Then, there exists $\delta>0$ such that $v(t)$ is not $\alpha$-concave  for $t\in(0,\delta)$.
\end{itemize}
\end{theorem}

Let $P$ be a solution to the porous medium type reaction-diffusion equation in terms of the pressure \eqref{pe} and be smooth to its support boundary, and let $w:=P^{\alpha}$ for $\alpha\in (0,1]$ and $w=\log P$ for $\alpha=0$. Next, we compute $\partial_t w_{11}$ at the origin $(0,0)$. To deal with the source term, we need to redesign some different initial datum from \cite{CW2024,ST2024}. By means of the equation \eqref{rde}, we have for $\alpha\in (0,1]$,
\begin{equation}\label{partialt11}
\begin{aligned}
    \partial_t w_{11}=&rw^{\frac{1}{\alpha}}w_{kk11}+\frac{2r}{\alpha}w^{\frac{1}{\alpha}-1}w_1w_{kk1}+\frac{r}{\alpha}(\frac{1}{\alpha}-1)w^{\frac{1}{\alpha}-2}w_1^2w_{kk}+\frac{r}{\alpha}w^{\frac{1}{\alpha}-1}w_{11}w_{kk}\\
    &+\frac{1}{\alpha}\big(1+r(1-\alpha)\big)\{(\frac{1}{\alpha}-2)(\frac{1}{\alpha}-1)w^{\frac{1}{\alpha}-3}w_1^2w_k^2\\
    &+(\frac{1}{\alpha}-1)w^{\frac{1}{\alpha}-2}w_{11}w_k^2+4(\frac{1}{\alpha}-1)w^{\frac{1}{\alpha}-3}w_1w_kw_{k1}\\
    &+2w^{\frac{1}{\alpha}-1}w_{1k}^2+2w^{\frac{1}{\alpha}-1}w_kw_{k11}\}+rw_{11}\big(\alpha G(w^{\frac{1}{\alpha}})+w^{\frac{1}{\alpha}}G'(w^{\frac{1}{\alpha}})\big)\\
    &+r|w_1|^2\big[(\frac{1}{\alpha}+1)w^{\frac{1}{\alpha}-1}G'(w^{\frac{1}{\alpha}})+\frac{1}{\alpha}w^{\frac{2}{\alpha}-1}G''(w^{\frac{1}{\alpha}})\big],
\end{aligned}
\end{equation}
and for $\alpha=0$,
\begin{equation}\label{partialt110}
\begin{aligned}
\partial_t w_{11}=&re^ww_{kk11}+2re^{w}w_1w_{kk1}+re^ww_1^2w_{kk}+re^ww_{11}w_{kk}+me^ww_1^2w_k^2\\
&+me^ww_1w_kw_{k1}+2me^ww_{1k}^2+2me^ww_kw_{k11}\\
&+rw_{1}^2\big(G''(e^w)e^{2w}+G'(e^w)e^w\big)+rw_{11} G'(e^w)e^w.
\end{aligned}    
\end{equation}

Let $B_\rho$ be the open ball of radius $\rho>0$ in $\mathbb{R}^N$ centered at the origin. 
\begin{lemma}\label{npcl}
 For each $\alpha\in[0,1]\backslash\{\frac{1}{2}\}$, there exists a $r>0$  and a smooth  positive function $w$ on $\overline{B_r(0)}$ such that 
 \begin{itemize}
    \item[(\expandafter{\romannumeral1})] $w_{11}(0)=0$, $w_{ii}(0)<0$ for $2\leq i\leq N$ and $w_{ij}(0)=0$ otherwise.
    \item[(\expandafter{\romannumeral2})] $\big(D^{2}w\big)<0$ on $\overline{B_{\rho}(0)}\backslash \{0\}$.
    \item[(\expandafter{\romannumeral3})] $\partial_t w_{11}>0$ from the right hand side \eqref{partialt11}-\eqref{partialt110} at the origin $(0,0)$.
\end{itemize}
\end{lemma}
\begin{proof}
The proof is divided into two cases. 

\noindent\textbf{\emph{Case 1, $\alpha\in[0,\frac{1}{2})$.}} Define 
\begin{equation}\label{w}w=c+x_1-x_1^4+\sum\limits_{i=2}^{N}(ax_1x_i^2-x_i^2-2a^2x_1^2x_i^2\big),\end{equation}
where the constants $a,c>0$ are determined later. By direct computations, we get 
\begin{align}
 w_{11}=&-12x_1^2-\sum\limits_{i=2}^{N}4a^2x_i^2,\notag\\
 w_{1i}=&w_{i1}=2ax_i-8a^2x_1x_i\quad \text{ for }i=2,...,N,\notag\\
 w_{ii}=&2a x_1-2-4x_1^2\quad \text{ for }i=2,...,N,\notag\\
 w_{ij}=&0\quad \text{for }i,j=2,...,N,\notag
\end{align}
Then, (\expandafter{\romannumeral1}) naturally holds.

\vspace{2mm}

For (\expandafter{\romannumeral2}) on the Hessian matrix $(D^2w)$, we calculate directly determinant of the $j$-$th$ leading principle minor $(D^2w)_j$ for $j=2,...,N$, that is   
\[\text{determinant of } (D^2w)_j=(-2)^{j}(6x_1^2+2a^2\sum\limits_{i=2}^Nx_i^2-a^2\sum\limits_{i=2}^{j}x_i^2)+\text{higher order terms}.\]
Due to
\[6x_1^2+2a^2\sum\limits_{i=2}^Nx_i^2-a^2\sum\limits_{i=2}^{j}x_i^2>0 \text{ for }x\neq0, \]
it holds by taking small $\rho>0$ that 
\[(-1)^j\big(\text{determinant of } (D^2w)_j\big)>0,\quad x\in \overline{B_{\rho}}\backslash\{0\}.\]
Hence, $(D^2w)$ is a strictly negative-definite matrix in $\overline{B_{\rho}}\backslash\{0\}$, so (\expandafter{\romannumeral2}) holds.

\vspace{2mm}

We turn to (\expandafter{\romannumeral3}), at the origin $(0,0)$, it follows from direct calculations that  
\begin{align}
&w=c,         &&w_{1}=1,         &&&w_{kk}=-2,\notag\\
&w_{kk1}=2a,  &&w_{1111}=-24,      &&&w_{kk11}=-8a^2.\notag
\end{align}
Then,  the right hand side of \eqref{partialt11} at the origin $(0,0)$ is 
\begin{align}R(a,c)=&-rc\big(24+8(N-1)\big)a^2+\frac{2r}{\alpha}c^{\frac{1}{\alpha}-1}a+\frac{1}{\alpha}\big(1+r(1-\alpha)\big)(\frac{1}{\alpha}-2)(\frac{1}{\alpha}-1)a^4\notag\\
&+r\big[(\frac{1}{\alpha}+1)c^{\frac{1}{\alpha}-1}G'(c^{\frac{1}{\alpha}})+\frac{1}{\alpha}c^{\frac{2}{\alpha}-1}G''(c^{\frac{1}{\alpha}})\big].\notag\end{align}
For $\alpha\in (0,\frac{1}{2})$, it holds by $a\gg1$, and $|G'(c^{\frac{1}{\alpha}})|+|G''(c^{\frac{1}{\alpha}})|<\infty$ for some $c>0$ that 
\begin{equation}\label{012}
\partial_t w_{11}(0,0)=R(a,c)>0.
\end{equation}

 In the case of $\alpha=0$, let $|G'(e^c)|+|G''(e^c)|<\infty$ for some $c>0$,  the right hand side of \eqref{partialt110} with sufficiently large $a>0$ yields
\begin{equation}\label{0}
\begin{aligned}
\partial_t w_{11}(0,0)=&-r\big(24+8(N-1)\big)\partial_t w_{11}(0,0)+4r(N-1)a-2r(N-1)a^2+ma^4\\
&+ra^2\big(G''(e^c)e^{2c}+G'(e^c)e^c\big)>0.
\end{aligned}
 \end{equation}
Hence,  (\expandafter{\romannumeral3}) is justified based on the above computations \eqref{012}-\eqref{0} for $\alpha\in[0,\frac{1}{2})$.

\noindent\textbf{\emph{Case 2, $\alpha=1$.}} Still for the function $w$  given by \eqref{w} with $\alpha=1$, (\expandafter{\romannumeral1}) and (\expandafter{\romannumeral2}) naturally hold as \textbf{\emph{Case 1.}}
By means of he condition \eqref{g0b} on $G$, we have 
\[\begin{aligned}
\partial_t w_{11}(0,0)=&2a+[2G'(c)+cG''(c)]-ca^2\big(24+8(N-1)\big)\\
\geq &2a-ca^2\big(24+8(N-1)\big)-Ac^{\gamma-1}.
\end{aligned}\]
 For $\gamma\geq 1$, let $c:=a^{-2}$, we take large $a>0$ such that 
\begin{equation}\label{ptw111}\partial_t w_{11}(0,0)\geq 2a-24-8(N-1)-Aa^{-2(\gamma-1)}>0.\end{equation}
For $\gamma\in(0,1)$, it holds by setting $c=:(Aa)^{\frac{-1}{1-\gamma}}$ that 
\begin{equation}\label{ptw112}\partial_t w_{11}(0,0)\geq2a-a^{2-\frac{1}{1-\gamma}}(A)^{\frac{-2}{1-\gamma}}\big(24+8(N-1)\big)-a>0 \end{equation}
for large $a>0$. (\expandafter{\romannumeral3}) is valid for $\alpha=1$ on account of \eqref{ptw111}-\eqref{ptw112}.

\vspace{2mm}

\noindent\textbf{\emph{Case 3, $\alpha\in(\frac{1}{2},1)$.}}     Redefine 
\[w=c+\frac{\alpha(\frac{3}{2}-\alpha)^{\frac{1}{2}}}{b(1-\alpha)}x_1-\frac{x_1^4}{12b^2}+\sum\limits_{i=2}^{N}\big(-b^2x_i^2+b(\frac{3}{2}-\alpha)^{\frac{1}{2}}x_1x_i^2-x_1^2x_i^2\big),\]
where $b>0$ is to be determined later. The each component of the Hessian matrix $(D^2w)$ is 
\begin{align}
w_{11}&=-\frac{x_1^2}{b^2}-2\sum\limits_{i=2}^{N}x_i^2\notag\\
w_{1i}&=w_{i1}=2b(\frac{3}{2}-\alpha)^{\frac{1}{2}}x_i-4x_1x_i\quad\text{ for }i=2,...,N,\notag\\
w_{ii}&=-2b^2+2b(\frac{3}{2}-\alpha)^{\frac{1}{2}}x_1-2x_1^2\quad\text{ for }i=2,...,N,\notag\\
w_{ij}&=0\quad \text{for }i,j=2,...,N,\notag
\end{align}
which verifies (\expandafter{\romannumeral1}) by taking the value at the origin.

\vspace{2mm}

Similar to the computations of \textbf{\emph{Case 1}}, the determinant of the $j$-th leading principal minor is given by 
\[\text{determinant of } (D^2w)_j=(-1)^{j}(2b^2)^{j-1}(\frac{x_1^2}{b^2}+2\sum\limits_{i=2}^Nx_i^2-2(\frac{3}{2}-\alpha)\sum\limits_{i=2}^{j}x_i^2)+\text{higher order terms}.\]
Since 
\[\frac{x_1^2}{b^2}+2\sum\limits_{i=2}^Nx_i^2-2(\frac{3}{2}-\alpha)\sum\limits_{i=2}^{j}x_i^2>0,\quad\alpha\in (\frac{1}{2},1),\quad x\neq0,\]
it concludes by taking small $\rho>0$ that 
\[(-1)^j\big(\text{determinant of } (D^2w)_j\big)>0,\quad x\neq0,\ |x|\leq \rho,\]
which proves (\expandafter{\romannumeral2}).

\vspace{2mm}

We are going to verify (\expandafter{\romannumeral3}).  Due to the explicit form of $w$,  it is direct to obtain  
\begin{align}
&w=c,         &&w_{1}=\frac{\alpha(\frac{3}{2}-\alpha)^{\frac{1}{2}}}{b(1-\alpha)},         &&&w_{kk}=-2b^2,\notag\\
&w_{kk1}=2b(\frac{3}{2}-\alpha)^{\frac{1}{2}},  &&w_{1111}=-\frac{2}{b^2},     &&&\ \ \ w_{kk11}=-4.\notag
\end{align}
Substituting these values into \eqref{partialt11} at the origin, for $|G'(c)|+|G''(c)|<\infty$ with some $c>0$, we attain
\begin{align}
\partial_t w_{11}=&r(-\frac{2}{b^2}-4(N-1))+2r\frac{\frac{3}{2}-\alpha}{1-\alpha}(N-1)\notag+\frac{1}{\alpha}\big(1+r(1-\alpha)\big)(\frac{1}{\alpha}-2)(\frac{1}{\alpha}-1)\big(\frac{\alpha^2(\frac{3}{2}-\alpha)}{b^2(1-\alpha)^2}\big)^{2}\notag\\
&+r\frac{\alpha^2(\frac{3}{2}-\alpha)}{b^2(1-\alpha)^2}\big[(\frac{1}{\alpha}+1)G'(c)+\frac{1}{\alpha}G''(c)\big]\notag\\
=&\frac{\mathcal{A}_{\alpha,m}}{b^2}-\frac{\mathcal{B}_{\alpha,m}}{b^4}+r(N-1)\big(\frac{2\alpha-1}{1-\alpha}\big)\notag
\end{align}
with \begin{align}\mathcal{A}_{\alpha,m}&:=-2r+r\frac{\alpha^2(\frac{3}{2}-\alpha)}{(1-\alpha)^2}\big[(\frac{1}{\alpha}+1)G'(c)+\frac{1}{\alpha}G''(c)\big],\notag\\
\mathcal{B}_{\alpha,m}&:=\frac{1}{\alpha}\big(1+r(1-\alpha)\big)(2-\frac{1}{\alpha})(\frac{1}{\alpha}-1)\big(\frac{\alpha^2(\frac{3}{2}-\alpha)}{(1-\alpha)^2}\big)^{2}.\notag\end{align}
Therefore, due to $r(N-1)\big(\frac{2\alpha-1}{1-\alpha})>0$, 
choosing large $b>0$ means that  (\expandafter{\romannumeral3}) holds.
\end{proof}

Inspired by \cite{CW2024,ST2024}, we are going to prove the concave non-preservation of  Theorem~\ref{npc} via constructing the counterexample. To be more precise, we construce $\alpha$-concave  initial pressure for $\alpha\in[0,1]\backslash\{\frac{1}{2}\}$, and then the $\alpha$-concavity of the pressure loses instantaneously.

\begin{proof}[\textbf{\emph{\underline{Proof of Theorem~\ref{npc}}}}]In Lemma~\ref{npcl}, we get a concave smooth positive function $w$ on $\overline{B_{\rho}}$. To begin with, we can smoothly extend $w$ on $\overline{B_\rho}$ to a smooth concave positive function $\tilde{w}$ on $B$. 
To this end, we first define a smooth cutoff function $\Phi:\overline{B}\to[0,1]$
\begin{equation*}
   \Phi:=\begin{cases}1\quad  &x\in\overline{B_{\frac{\rho}{2}}},\\
   0\quad &x\in\overline{B}\backslash B_{\frac{3\rho}{4}}.
   \end{cases}
\end{equation*}
There exists a constant $C>0$ such that 
\begin{equation}\label{dpsi}
    |D\Phi|+|D^2\Phi|\leq C,\quad\text{in }\overline{B}.
\end{equation}
In addition, we define a radial concave function $F:\overline{B}\to[0,\infty)$ as $F(x)=f(|x|)$, $f(r):[0,1]\to [0,\infty)$ is smooth concave function. 

For the case of $\alpha\in(0,1]$, $f$ satisfies
\begin{equation*}
    f(r):=\begin{cases}1\quad &0\leq r\leq \frac{\rho}{4},\\
        1-r^2 \quad &\frac{\rho}{2}\leq r\leq 1.
    \end{cases}
\end{equation*}
One can immediately check that  $F(x):=f(|x|)$ is smoothly concave on $B$ and vanishes on $\partial B$. 
The direct computations yield 
\begin{equation*}
    \big( D^2F \big)\leq -2\textbf{Id},\quad\text{in }B\backslash B_{\frac{\rho}{2}}
\end{equation*}
where $\textbf{Id}$ is the N-order unit matrix. 
In this way, we can define the expected initial data $\tilde{w}:B\to[0,\infty)$ by 
\[\tilde{w}=A F+\Phi w,\]
where the constant $A>0$ is determined later. Then, the Hessian matrix of $\tilde{w}$ can be given by 
\begin{equation*}
\begin{aligned}
    \big(D^2\tilde{w}\big)=A\big(D^2 F\big)+\Psi \big(D^2w\big)+w(D^2\Psi)+2D\Psi (D w)^\mathrm{T}.
\end{aligned}
\end{equation*}
Since $\Psi$ is identical to 1 in $B_{\frac{\rho}{2}}$, it holds 
\[\big(D^2\tilde{w}\big)=A\big(D^2 F\big)+\Psi \big(D^2w\big)<0,\quad \text{in }B_{\frac{\rho}{2}}\backslash\{0\}.\]
Furthermore, on $\overline{B}\backslash B_{\frac{\rho}{2}}$,   combining with \eqref{dpsi} yields
\[\big(D^2\tilde{w}\big)\leq (C-2A)\textbf{Id}<0,\]
by choosing sufficiently large $A>0$. 

Let $\lambda_1(x,t)$ be the largest eigenvalue of the matrix $\big(D^2 P^\alpha \big)$ and be smooth in the neighborhood of origin and $\big(D^2P^2\big)=\big(D^2 w\big)$ at the origin. From Lemma \ref{npcl}, we know that $\lambda_1(0,0)=0$. All spatial derivatives of $P^\alpha$ in $(x,t)=(0,0)$ coincide with those of $w$ at the origin, then it follows from (\expandafter{\romannumeral3}) of Lemma \ref{npcl} that \[\partial_t\lambda_1(0,0)=\partial_t(P^\alpha)_{11}(0,0)=\partial_t w_{11}(0,0)>0.\] Therefore, there are exits $\delta>0$ such that
\[\lambda_1(0,t)>0,\quad t\in(0,\delta).\]
This  means that  $P^\alpha$ for $\alpha\in (0,1]\backslash\{\frac{1}{2}\}$ is not concave in the time interval $(0,\delta)$ .

Finally, for $\alpha=0$, define 
\begin{equation*}
    f(r)=\begin{cases}\log(1-\frac{\rho}{2})+\frac{1}{4},\quad&0\leq r\leq \frac{\rho}{4},\\
    \log(1-r),\quad &\frac{\rho}{2}\leq r\leq 1.
    \end{cases}
\end{equation*}
Let $F(x):=f(|x|)$ on $B$ and $\tilde{w}:=AF+\Psi w$. The similar calculations to the previous case $\alpha\in (0,1]\backslash\{\frac{1}{2}\}$ yield
\[\big(D^2\tilde{w}\big)\leq 0.\]
Using Lemma \ref{npcl} derives
\[\partial_t \lambda_1(0,0)>0 \text{ and }\lambda_1(0,0)=0.\]
The rest is same as in the case of $\alpha\in(0,1]\backslash\{\frac{1}{2}\}$.
\end{proof} 
\subsection{Non-degeneracy estimate}
Under the $\frac{1}{2}$-concave results (Lemma \ref{12concave}), we can estimate the boundary behavior of the pressure. In the following, we show that the pressure $P$ degenerates linearly to its spatial support boundary.   We first give the upper bound estimate of the pressure gradient. Compared with previous work \cite{DHL2001}, to overcome the difficulty caused by non-monotonicity of the source term $uG(P)$, we introduce the new variable \[ v:=e^{-G(0)t}u.\] 
\begin{lemma}\label{ue}
    If the initial pressure $P_0$ satisfies \eqref{ini}-\eqref{ini2} and the second estimate of \eqref{ini3}, and assume that the pressure $P$ for the porous medium type reaction-diffusion equation \eqref{rde}  is smooth to it support boundary with the condition \eqref{concavityc} on $G$. Then, it holds for all $r:=m-1>0$ that 
\[|\nabla P|\leq Ce^{rG(0)t},\ t\geq 0.\]
\end{lemma}
\begin{proof} 
The  new variable $$v:=e^{-G(0)t}u$$ solves
\[\partial_t v=e^{rG(0)t}\Delta v^m+v[G(e^{rG(0)t}\tilde{P})-G(0)],\]
and the corresponding pressure $\tilde{P}:=e^{rG(0)t}v^{m-1}$ satisfies
\[\partial_t\tilde{P}=e^{rG(0)t}[r\tilde{P}\Delta \tilde{P}+|\nabla \tilde{P}|^2+r\tilde{P}(G(e^{rG(0)t}\tilde{P})-G(0))].\]

We are going to use the comparison principle to achieve our goal.  To this end, we assume 
\[X=\frac{\tilde{P}_k^2}{2}-\epsilon t,\quad t\geq0.\]
Since
\[
\begin{aligned}X_t=&\tilde{P}_k\tilde{P}_{kt}-\epsilon
\\
=&e^{rG(0)t}\tilde{P}_k[r\tilde{P}_k\tilde{P}_{ii}+r\tilde{P}\tilde{P}_{kii}+2\tilde{P}_{ki}\tilde{P}_i]\\
&+re^{rG(0)t}\tilde{P}_k^2\big(G(e^{rG(0)t}\tilde{P})-G(0)+e^{rG(0)t}\tilde{P}G'(e^{rG(0)t}\tilde{P})\big)-\epsilon
\end{aligned}\]
and 
\[X_k=\tilde{P}_i\tilde{P}_{ik},\ X_{kk}=\tilde{P}_i\tilde{P}_{ikk}+\tilde{P}_{ik}^2,\]
we have
\[
\begin{aligned}X_t=&e^{rG(0)t}[r\tilde{P} X_{kk}-r\tilde{P}\tilde{P}_{ik}^2+r\tilde{P}_i^2\tilde{P}_{kk}+2\tilde{P}_kX_k]\\
&+re^{rG(0)t}|\tilde{P}_k|^2\underbrace{\big(G(e^{rG(0)t}\tilde{P})-G(0)+e^{rG(0)t}\tilde{P}G'(e^{rG(0)t}\tilde{P})\big)}_{\leq0}-\epsilon.\end{aligned}\]
\noindent\textbf{\emph{Case 1, interior estimate.}} At an interior maximum point $(x_0,t_0)$ of $X$, we derive
\begin{equation*}
\tilde{P}X_{kk}\leq 0
\end{equation*}
and 
\begin{equation*}
X_k=0.
\end{equation*}
Hence, we obtain 
\begin{equation}\label{pkk}X_t\leq e^{rG(0)t}\tilde{P}_i^2\tilde{P}_{kk}-\epsilon. \end{equation}
It is enough to show 
\begin{equation}\label{pkk1}\tilde{P}_{kk}\leq 0,\quad k=1,...,N.
\end{equation}
By rotating the coordinates, we can assume that \begin{equation*}\label{rc1}\tilde{P}_N\geq 0,\quad \tilde{P}_i=0,\ i=1,2...,N-1\end{equation*}at the point $(x_0,t_0)$.  Next, $\tilde{P}_N$ is divided into two cases as $\tilde{P}_N>0$ and $\tilde{P}_N=0$.

\vspace{2mm}

For the case of $\tilde{P}_N>0$ at $(x_0,t_0)$, the direct computation 
\[X_N=\tilde{P}_N\tilde{P}_{NN}=0\text{ at }(x_0,t_0)\] implies \[\tilde{P}_{NN}=0\text{ at }(x_0,t_0).\]
The $\frac{1}{2}$-concavity of $\tilde{P}=e^{-rG(0)t}P$ (Theorem~\ref{concavityrde}) yields 
\begin{equation}\label{1/2conca}\tilde{P}\tilde{P}_{ij}V^iV^j\leq \frac{1}{2}\tilde{P}_i\tilde{P}_jV^iV^j.\end{equation}
Hence, we further get 
 \[\tilde{P}_{ii}\leq 0,\quad 1\leq i\leq N-1 \text{ at }(x_0,t_0)\]
 by taking $V^i=\delta_{ik}$. 

\vspace{2mm}

For the case of $\tilde{P}_N=0$ at $(x_0,t_0)$, the $\frac{1}{2}$-concavity \eqref{1/2conca} yields 
\[\tilde{P}_{ii}\leq \frac{|\partial_i \tilde{P}|^2}{2\tilde{P}}=0,\quad i=1,...,N \text{ at }(x_0,t_0).\]

Finally, it concludes that \eqref{pkk} holds by means of the previously verified \eqref{pkk1}.
 
 \vspace{1mm}
 
\noindent\textbf{\emph{Case 2, boundary estimate}}    Assumde that its maximum occurs at a support boundary point $(x_0,t_0)$ and  
\[\tilde{P}_N>0,\text{ and }\tilde{P}_i=0,\ i=1,...,N-1\text{ at }(x_0,t_0).\]
Due to $\tilde{P}(x_0,t_0)=0$, it holds 
\[X_t\leq e^{(m-1)G(0)t}[\tilde{P}_N^2\tilde{P}_{kk}+2r\tilde{P}_NX_N]-\epsilon\text{ at }(x_0,t_0).\]
On account of $\tilde{P}_N>0$ at the maximum point of $X$, we have 
\[X_N\leq 0\text{ at }(x_0,t_0).\]
It further follows
\[X_N\tilde{P}_N\leq 0\text{ at }(x_0,t_0).\]
We are going to show that $\tilde{P}_{kk}\leq 0,\ k=1,...,N-1.$
Since 
\[X_N=\tilde{P}_N\tilde{P}_{NN}\leq 0,\]
we get 
\[\tilde{P}_{NN}\leq 0\ \text{at }(x_0,t_0).\]
Similar to the case of interior estimate, it holds by means of  $\frac{1}{2}$-concavity that 
\[\tilde{P}_{ii}\leq 0,\ i=1,...,N-1\ \text{at }(x_0,t_0).\]

In conclusion, combining \textbf{\emph{Case 1, interior estimate}}  and \textbf{\emph{Case 2, boundary estimate}}, it holds at the maximum point that 
\[X_t\leq -\epsilon,\]
which is contradicted with $X_t\geq0$ and  yields that $X$ attain its maximum at the initial pressure. Hence, we obtain
\[|\nabla P|\leq e^{rG(0)t}C,\quad t\geq0. \]
\end{proof}
We are going to verify the Aronson-B\'enilan estimate, which shows the semi-harmonic property of the pressure.

\begin{lemma}[Aronson-B{\'e}nilan estimate]\label{lAB}
Under the initial assumptions \eqref{ini}-\eqref{ini2} and the condition \eqref{abc} on $G$, then, for the pressure of the porous medium type reaction-diffusion equation \eqref{rde}, it holds for all $r:=m-1>0$ that 
\begin{equation}\label{ab2}
\Delta P+G(P)\geq -\frac{K}{1+rKt}\quad \text{in }\mathcal{D}'(\mathbb{R}^n\times \mathbb{R}_+).
\end{equation}
\end{lemma}
\begin{proof}Let $\Delta \times \eqref{pe}+G'(P)\times\eqref{pe}$ and set $\omega:=\Delta P+G(P)$, then we have
\begin{equation}\label{ab1}
\begin{aligned}
\partial_t \omega=&rP\Delta \omega+2m\nabla P\cdot\nabla \omega+r\omega^2-(G(P)-PG'(P))r\omega\\&-G'(P)|\nabla P|^2+2\nabla^2 P:\nabla^2 P\\
\geq&rP\Delta \omega+2m\nabla P\cdot\nabla \omega+r\omega^2-(G(P)-PG'(P))r\omega.
\end{aligned}
\end{equation}
Set $|f(x)|_-:=-\min\{f(x),0\}$ and  multiplying~\eqref{ab1} by $sign(|\omega|_-)$,  thanks to Kato's inequality,  it holds in the sense of distribution that 
\begin{equation*}
\begin{aligned}
\partial_t |\omega|_-\leq& rP\Delta |\omega|_-+2m\nabla P\cdot\nabla |\omega|_--r|\omega|_-^2-(G(P)-PG'(P))r|\omega|_-\\
\leq & rP\Delta |\omega|_-+2m\nabla P\cdot\nabla |\omega|_--r|\omega|_-^2.
\end{aligned}
\end{equation*}

Since $f(t)=\frac{K}{1+rKt}$ with $f(0)=K$ solves $\frac{d}{dt} f=-rf^2$,  it concludes by the comparison principle that $|\omega(\cdot,t)|_-\leq f(t)$  in the sense of distribution, and the result \eqref{ab2} holds.
\end{proof}

\begin{lemma}\label{alphat}
Assume that the initial pressure $P_0$ satisfies the assumptions \eqref{ini}-\eqref{ini3}
and $G$ has the properties \eqref{concavityc} and \eqref{abc}-\eqref{abc1}. If the pressure $P$ is smooth to its support boundary, then, the lower bound estimate holds
\[\big[1+(t+\frac{2}{Kr+2})rG(P)\big]P+(t+\frac{2}{Kr+2})(1+\frac{r}{2})|\nabla P|^2\geq \frac{c}{Kr+2}e^{-rLt}\quad \text{on }\{P>0\}.\]
\end{lemma}
\begin{proof}[\bf Proof]
Let \[F=(t+\alpha)P_t+P\] 
with $\alpha$ being given later, which, by direct computations, solves
\begin{equation*}
    \begin{aligned}
 F_t=&2P_t+(t+\alpha)P_{tt}\\
 =&2P_t+(t+\alpha)[rP\Delta P+|\nabla P|^2+rPG(P)]_t\\
 =&2P_t+(t+\alpha)[rP_t\Delta P+rP\Delta P_t+2\nabla P_t\cdot\nabla P+rP_tG(P)+rP P_tG'(P)]\\
 =&2P_t-rP(\Delta P+G(P))+rF(\Delta P+G(P))+rP\Delta F-rP\Delta P\\
 &+2\nabla F\cdot\nabla P-2|\nabla P|^2+rPG'(P)(F-P)\\
 =&rP\Delta F+2\nabla F\cdot\nabla P+rF(\Delta P+G(P)+PG'(P))+rP(G(P)-PG'(P))\\
 \geq&rP\Delta F+2\nabla F\cdot\nabla P+rF(\Delta P+G(P)+PG'(P)+L)-rLF\\
  \geq&rP\Delta F+2\nabla F\cdot\nabla P-rLF,
\end{aligned}
\end{equation*}
where we used the pressure equation \eqref{pe} in second line, and the condition \eqref{concavityc} on $G(\cdot)$ in the seventh line and $\Delta P+G(P)+PG'(P)\geq -K-(L-K)\geq -L$, derived from both \eqref{ab2} and \eqref{abc1}, in the last line. Therefore, $\tilde{F}:=Fe^{rLt}+\epsilon t$ satisfies the differential inequality 
\begin{equation}\label{piee}
  \tilde{F}_t\geq rP\Delta \tilde{F}+2r\nabla P\cdot\nabla \tilde{F}+\epsilon.  
\end{equation}

We are going to show that $\tilde{F}$ does not attach $0$ in any finite time. If not, there exists $(x_0,t_0)$ with $t_0>0$ such that $\tilde{F}(x_0,t_0)=0$ and $\tilde{F}(x,t)>0$ in $\{P>0\}\cap\{t<t_0\}$.
By the maximum principle, the minimum of $\tilde{F}$ can not be obtained in the interior of the set $\{P>0\}$. Let $(x_0,t_0)$ be the minimum of  $\tilde{F}$ on the free boundary. By rotating the coordinate, we can assume that at the point $(x_0,t_0)$, $P_N>0$, and $P_i=0$ for all $i=1,...,N-1$. Then, it holds
\[\tilde{F}_N\geq 0\quad \text{at }(x_0,t_0).\]
Therefore, $\nabla P\cdot \nabla \tilde{F}\geq 0$ at $(x_0,t_0)$ and hence by \eqref{piee}, we obtain 
\begin{equation*}
\tilde{F}_t\geq\epsilon\quad\text{at }(x_0,t_0),
\end{equation*}
which is in contradiction with the minimum point $(x_0,t_0)$. Hence, 
\begin{equation}\label{ppie}\tilde{F}>0\quad\text{on } \{P>0\}.\end{equation}
The non-negative property $\tilde{F}\geq0$ yields
\begin{equation*}\label{pie}
  \tilde{F}_t\geq rP\Delta \tilde{F}+2r\nabla P\cdot\nabla \tilde{F}+\epsilon.  
\end{equation*}
Furthermore, following the similar arguments to verify \eqref{ppie}, the minimum of $\tilde{F}$ is attained in the initial data, it holds 
\begin{equation*}
\tilde{F}\geq \tilde{F}(0).
\end{equation*}
Let $\epsilon\to 0$, it follows
\begin{equation}\label{f0}
Fe^{rLt}\geq F(0).
\end{equation}

In addition, the compatibility conditions yields
$$F(0)=(\alpha P_t+P)|_{t=0}=\alpha r P_0(\Delta P_0+G(P_0))+\alpha|\nabla P_0|^2+P_0$$
Since $$\Delta P_0+G(P_0)\geq -K,$$ and we set $\alpha=:\frac{2}{Kr+2}$, it holds
\[F(0)\geq (1-\alpha K)P_0+\alpha |\nabla P_0|^2=\frac{1}{Kr+2}(P_0+|\nabla P_0|^2)\geq \frac{c}{Kr+2}.\]
Combining with \eqref{f0} proves
\[F\geq \frac{c}{Kr+2}e^{-rLt}\quad\text{on }\{P>0\}.\]
Furthermore, it holds by the pressure equation \eqref{012} that 
\begin{equation}\label{a1}
(t+\frac{2}{Kr+2})rP(\Delta P+G(P))+(t+\frac{2}{Kr+2})|\nabla P|^2+P\geq \frac{c}{Kr+2}e^{-Krt}\quad \text{on }\{P>0\}.
\end{equation}
On the other hand, $\frac{1}{2}$-concavity of the pressure $P$ (Theorem \ref{concavityrde}) yields 
\begin{equation}\label{a2}
P\Delta P\leq \frac{1}{2}|\nabla P|^2.
\end{equation}
Taking \eqref{a2} into \eqref{a1}, it holds 
\[\big[1+(t+\frac{2}{Kr+2})rG(P)\big]P+(t+\frac{2}{Kr+2})(1+\frac{r}{2})|\nabla P|^2\geq \frac{c}{Kr+2}e^{-rLt}\quad \text{on }\{P>0\}.\]
\end{proof}

\section{Hele-Shaw problem}\label{hsp}
Since the pressure $P$ solves an elliptic equation \eqref{hss} for the Hele-Shaw problem \eqref{hs}, does not solves  the degenerate parabolic equation \eqref{pe} anymore. We are going to  use a completely different approach to prove the sharp power concavity of the pressure for the Hele-Sahw problem \eqref{hss}.

To begin with, we are going to state the existence and uniqueness of the weak solution to the Hele-Shaw problem $\eqref{hs}$, which can be extracted from previous works \cite{DN2023,PERTHAME2014}.
\begin{theorem}\label{eu}
    Assume that the initial data $u_0$ satisfies $0\leq u_0\leq 1$ and is compactly supported, then there exists a unique solution $(u,P)$ to the Hele-Shaw problem $\eqref{hs}$.
\end{theorem}
\begin{proof}
We note that $0\leq u_0^{\gamma}\leq 1$ for any $\gamma>0$ due to $0\leq u_0\leq 1$. Then, from \cite{DN2023}, through the Hele-Shaw asymptotics for \eqref{rde} as $m\to\infty$, the weak solution for $\eqref{hs}$ exists. The uniqueness is verified in \cite{PERTHAME2014} by means of the Hilbert duality method.
\end{proof}
\begin{remark}
 It should be noticed that, under the initial integral assumptions as in Theorem \ref{eu}, the authors use the special structure of the porous medium type equation to obtain a strong convergence of the pressure and its gradient and also deal with the nonlinear reaction term by means of its concavity in \cite{DN2023}. Higher regularities, even including BV regularity, are no longer necessary.
\end{remark}

Let $(u,P)$ be the weak solution to $\eqref{hs}$ under the initial assumptions as in Thoerem~\ref{eu}, thanks to \cite[Theorem 1.1]{DN2023}, we know that  there exists $\Omega(t)$ for  $t\geq0$ with 
\[\Omega(0):=\Omega_0\]
such that  
    \begin{equation}\label{r1}
    P>0,\quad u=1,\quad \text{in }\Omega(t),
    \end{equation}
    \begin{equation}\label{r2}
        \Delta P+G(P)=0,\quad \mathcal{D}'(\text{Int}(\Omega(t))),
    \end{equation}
    \begin{equation}\label{r3}
  P=0,\quad u(t)=e^{G(0)(t-s)}u(s)\text{ with }0\leq s\leq t\quad\text{a.e. in }\mathbb{R}^N\backslash \Omega(t),      
    \end{equation}
where $\text{Int}(\Omega(t))$ is the interior of $\Omega(t)$. If $P$ is smooth to the boundary, these results \eqref{r1}-\eqref{r3} imply  \begin{equation*}
 \Omega(t)=\{x\in\mathbb{R}^n: P(x,t)>0\}\text{ for all }t\geq0\quad \text{and } \overline{\Omega(t)}=\overline{\{u=1\}}.
\end{equation*}
Furthermore, due to the complementarity relationship~$\eqref{hs}_2$ under the same assumption of smoothness to its support boundary for the pressure $P$, we have
\begin{equation}\label{hss}
\begin{cases}
-\Delta P=G(P),\quad &\text{in }\Omega(t),\\
P=0,\quad&\text{on }\partial\Omega(t),
\end{cases}
\end{equation}
for all $t\geq 0$.

\paragraph{Sharp power concavity of initial pressure.}

\vspace{2mm}

From the above discussions, we learn that  the initial pressure solves an elliptic equation with Dirichlet boundary condition~\eqref{hs0}, the first is to study the sharp power concavity of the initial pressure. 
\begin{lemma}
    Assume that the initial domain $\Omega_0$ is an convex set with smooth boundary. Then, the initial pressure $P_0$, solving the elliptic problem \eqref{hs0},  is $\alpha$-concave for any $\alpha\in[0,\frac{1}{2}]$, and  $\frac{1}{2}$ is the largest index for power concavity.
\end{lemma}
\begin{proof}Let $p=\frac{1}{2}$ and $h=G$ in Theorem \ref{ect}, we have 
\begin{equation*}g(t)=
 tG(t^2).
 \end{equation*}
Due to \eqref{hsc}, the second order derivative satisfies 
\begin{equation*}g''(t)=6tG'(t^2)+4t^3G''(t^2)\leq 0\quad t\geq0.
 \end{equation*}
Then,  by Theorem \ref{ect}, $P_0$ is $\frac{1}{2}$-concave in $\Omega_0$. Lemma~\ref{cam} further yields that  $P_0$ is $\alpha$-concave for any $\alpha\in[0,\frac{1}{2}]$.

\vspace{2mm}

Next, we choose a sequence of smooth domains to show that the index $\frac{1}{2}$ is the upper bound for the concavity of the solution to the elliptic equation \eqref{hs0} on the smooth domain. To this end, we define $x\in\mathbb{R}^N$,  $x_N=x\cdot e_N$ with $e_N=(0,...,0,1)$, and $x'=x-x_N$, the cone is defined by 
\[K:=\{x\in \mathbb{R}^N: x_n> a |x'|,\ a>0\},\]
where $a>0
$ is determined later. Let $B_1:=\{x\in \mathbb{R}^N: |x|< 1\}$, then the desired   open domain is given by 
\[\Omega_\infty:=K\cap B_1.\]
We give a sequence of smooth convex domains satisfying
\begin{equation}\label{omegak}\Omega_\infty\subset...\subset \Omega_k\subset\Omega_{k-1}\subset...\subset\Omega_1\subset B_1 
\end{equation}
with 
\[0\in\partial\Omega_k\quad\text{for } k=1,2,3.... \text{ and } \lim\limits_{k\to\infty}\Omega_k=\cap_{k=0}^{\infty}\Omega_k=\Omega_\infty.\]
Let $P_{k}$ solve the elliptic problem with Dirichlet boundary codnition
\begin{equation}\label{omegake}
\begin{cases}
-\Delta P_{k}=G(P_{k}),\quad &\text{in }\Omega_k,\\
P_{k}=0,\quad&\text{on }\partial\Omega_k.
\end{cases}
\end{equation}
If $w$ is a solution of the Dirichlet type elliptic equation 
\begin{equation*}
\begin{cases}-\Delta w=G(0),\quad &\text{in }B_1,\\w=0,\quad&\text{on }\partial B_1,\end{cases}\end{equation*}
then
\[0\leq w=\frac{G(0)}{2N}(1-|x|^2)\leq \frac{G(0)}{2N},\quad \text{in }B_1. \] 
It also holds by the standard comparison principle that
\begin{equation}\label{pkb}0\leq P_{k}\leq w\leq \frac{G(0)}{2N}\quad \text{in }\Omega_\infty \text{ for any }k\geq 0.\end{equation}
Then, by means of the standard interior Shauder estimate in \cite{GT2001}, for any open subset $\mathcal{A}\subset \Omega_\infty$, we have
\[\|P_{k}\|_{\mathcal{C}^{1,\gamma}(\mathcal{A})}\leq C,\]
where $C$ is a positive constant only depending on $\mathcal{A}$.  By the Sobolev compactness embedding \cite{ELC2010}, after extraction of a subsequence (still denoted by $P_{k}$), it holds
\[P_{k}\to P_{\infty},\quad \text{in }\mathcal{C}_{loc}^{1}(\Omega_\infty),\ \text{ as }k\to\infty.\]
In addition, due to \eqref{omegak}, the comparison principle for the equation \eqref{omegake} yields 
\begin{equation}\label{mr}
0\leq P_{k+1}\leq P_{k}\quad\text{in }\Omega_\infty\text{ for } k=1,2,....
\end{equation}
By means of Dini's theorem, \eqref{pkb} and \eqref{mr} yield 
\begin{equation*}
    \|P_{k}-P_\infty\|_{\mathcal{C}(\Omega_\infty)}\to0\quad\text{as }k\to\infty.
\end{equation*}
It is immediate to conclude that  $P_{\infty}$ is the unique solution of the elliptic problem
\begin{equation}\label{pinftye}
\begin{cases}
-\Delta P_{\infty}=G(P_{\infty}),\quad &\text{in }\Omega_\infty,\\
P_{\infty}=0,\quad&\text{on }\partial\Omega_\infty.
\end{cases}
\end{equation}

We randomly fix  $\alpha\in(\frac{1}{2},1]$ and suppose that the solution $P$ of the elliptic problem \eqref{hs0} on the any given smooth convex domain is $\alpha$-concave. By the definition of $\alpha$-concavity, for $y=0\in \partial\Omega_\infty\cap\partial\Omega_k$, $k=1,2...$, and any  $t\in (0,1)$, $z\in \Omega_\infty$, we have $(1-t)y+tz=tz\in\Omega_\infty$ and
\begin{equation*}
P_{k}^{\alpha} (tz)=P_{k}^{\alpha} ((1-t)y+tz)\geq (1-t)P_{k}^{\alpha}(y)+tP_{k}^{\alpha}(z)=tP_{k}^{\alpha}(z)>0.
\end{equation*}
Let $k\to\infty$, we get 
\begin{equation}\label{pac}
P_{\infty}^{\alpha} (tz)\geq tP_{\infty}^{\alpha}(z)>0.
\end{equation}

Let  $v:=x_n^2-a^2|x'|^2 $ in $\Omega_\infty$, choosing $a\geq \sqrt{\frac{2+G(0)}{2N}}$, then 
\[\Delta(v-P_{\infty})=2(1-a^2N)+G(P_{\infty})\leq -2(a^2N-1)+G(0)\leq0. \]
Since $v\geq 0$ on $\partial\Omega_\infty$, the compsrison principle yields 
\[P_{\infty}\leq v\quad \text{in }\Omega_\infty.\]
We take $z=\frac{1}{2}e_N$ for \eqref{pac}, then it holds 
\begin{equation}\label{p0}P_{\infty}(\frac{1}{2}e_N)\leq\frac{P_{\infty}(t\frac{1}{2}e_N)}{t^{\frac{1}{\alpha}}}\leq \frac{v(t\frac{1}{2}e_N)}{t^{\frac{1}{\alpha}}}=\frac{t^{2-\frac{1}{\alpha}}}{4}\quad\text{ for any } t\in(0,1).\end{equation}
One can directly compute
\[\liminf\limits_{t\to0^+}\frac{t^{2-\frac{1}{\alpha}}}{4}=0\quad\text{ for }\alpha\in(\frac{1}{2},1],\]
which implies  the inequality  \eqref{p0} can not hold if $t$ is small enough. The contradiction appears because $P_{\infty}(\frac{1}{2}e_N)>0$ can be derived by the strong maximum principle for the elliptic problem~\eqref{pinftye}.

\end{proof}

\paragraph{Preservation of $\frac{1}{2}$-concavity.} This part is devoted to proving the preservation of $\frac{1}{2}$-concavity for the pressure of the Hele-Shaw problem \eqref{hs}. The key is to obtain the convexity for the spatial support of the pressure.
\begin{lemma}Assume that the initial density $0\leq u_0\leq 1$ is compactly supported and  the support of the initial pressure $P_0$ given by $\Omega_{0}:=\{u_0=1\}$ is convex with smooth boundary, and suppose that the pressure $P$ for the Hele-Shaw problem \eqref{hs} is smooth to its support boundary, then $P$ is $\alpha$-convex all the time for any  $\alpha\in[0,\frac{1}{2}]$. 
\end{lemma}
\begin{proof}At any fixed time, the stiff pressure $P$ solves an elliptic equation \eqref{hss}, the concavity of this pressure is determined by the shape of the spatial domain $\Omega(t)$. 
Assume that the initial pressure $P_0$ is $\frac{1}{2}$-concave and the stiff pressure $P$ is smooth to its support boundary.

If $P$ is not $\frac{1}{2}$-concave all the time,  there exists some $\delta>0$, the time $t_0$,  the  unit vector $\alpha'\in\mathbb{R}^N$ and $x_0\in \overline{\Omega(t_0)}$ such that 
\begin{equation}\label{x_0t_0}(\sqrt{P})_{\alpha'\alpha'}=\delta\quad\text{at }(x_0,t_0).\end{equation}
And furthermore, it holds for any unit vector $\alpha\in\mathbb{R}^N$ that 
\begin{equation}\label{spb}(\sqrt{P})_{\alpha\alpha}=\frac{P_{\alpha\alpha}}{2\sqrt{P}}-\frac{P_\alpha^{2}}{4P^{\frac{3}{2}}}\leq \delta\quad \text{in }\overline{\Omega(t)}\times\{0\leq t\leq t_0\}.\end{equation}

Without loss of generality, assume that $e_N=(0,...,1)$ is the normal vector at $(y_0,s_0)$ for any $y_0\in\partial\Omega(s_0)$.  We parameterize the nearby of $x_0$ for the surface $\partial\Omega(s_0)$ by $(x',\gamma(x',s))$ with
\[P_{\alpha}(y_0,s_0)=\gamma_{\alpha}(y_0',s_0)=0 \] for all unit vector satisfying $\alpha\bot e_N$, 
and Hopf's lemma for elliptic problem \eqref{hss} yields 
\[c_0\leq P_N(x_0,s_0)\leq \frac{1}{c_0}\text{ for some }c_0>0. \]
By means of  l'Hospital's rule, we have
\begin{equation}\label{pii}
\begin{aligned}
0=\lim\limits_{x\to y_0}2\delta\sqrt{P(x,s_0)}\geq& \lim\limits_{x\to y_0}[P_{\alpha\alpha}-\frac{P_\alpha^2}{2P}]\\
=&\lim\limits_{x\to y_0}[P_{\alpha\alpha}-\frac{P_{N\alpha}P_\alpha}{P_N}] =P_{\alpha\alpha}(y_0,s_0).
\end{aligned}
\end{equation}
Thanks to
\[P(x',\gamma(x',s),s_0)=0,\quad x'\in\mathbb{R}^{N-1}\cap B_{\varepsilon}(0)\text{ for some }\varepsilon>0,\]
we get 
\[P_{\alpha\alpha}+2P_{N\alpha}\gamma_\alpha(x',s)+P_{NN}|\gamma_\alpha(x',s)|^2+P_N\gamma_{\alpha\alpha}(x',s)=0.
\]
Hence, at $(y_0,s_0)$, it holds for all unit vector satisfying $\alpha\bot e_N$ that 
\[\gamma_{\alpha\alpha}(x_0',s_0)=-\frac{P_{\alpha\alpha}}{P_N}\geq 0\]
which implies that $\Omega(s_0)$ is convex for all $0\leq s_0\leq t_0$. Using Theorem \ref{ect} for the elliptic equation with Dirichlet boundary condition, $P(t)$ is $\frac{1}{2}$-concave on $\Omega(t)$ for $0\leq t\leq t_0$, which is contradicted with \eqref{x_0t_0}.  Hence, the pressure $P$ is $\frac{1}{2}$-concave all the time.

Finally, Lemma \ref{cam} verifies that $P$ is $\alpha$-concave for $0\leq \alpha\leq \frac{1}{2}$ all the while.
\end{proof}
\section{Conclusions and perspectives}\label{cp} We investigate and compare the geometric properties of the porous medium type reaction-diffusion equation \eqref{rde} and the corresponding Hele-Shaw problem \eqref{hs}. Precisely, we proved that $\frac{1}{2}$ is the sharp index for the preservation of the concavity over time for the pressure of these two free boundary problems. The key distinction lies in the role of this index: for the porous medium-type reaction-diffusion equation \eqref{rde}, the index $\frac{1}{2}$ uniquely guarantees the preservation of concavity over time. In contrast, for the Hele-Shaw problem \eqref{hs}, $\frac{1}{2}$ acts as an upper bound on the index to preserve concavity as the time progresses. Furthermore, using the concave property, we establish non-degenerate estimates for the porous medium type reaction-diffusion equation, which means that the spatially Lipschitz continuity is sharp. 

\vspace{2mm}

Based on the results of this paper, the further problem is to prove that the pressure is smooth to its support boundary as \cite{DHL2001,DL2004} for the two free boundary problems, which is the a priori assumption in the processes of paper. Furthermore, if we take the nutrient into account for the two free boundary problems \cite{PERTHAME2014}, how are the geometric properties including the concavities of the pressure and free boundary regularities? There were some results~\cite{JKT2023,KL2023,CJK2023} on the free boundary regularities for the Hele-Shaw problem of tumor growth with/without nutrient diffusion. For the tumor growth model with nutrient of porous medium type, the study of free boundary and its geometric property is a completely open problem. 
\appendix
\section{Some preliminary results} 
The following theorem, a simple version of \cite{BS2013},  gives the sufficient condition of  power concavity of the solution to the elliptic equation with Dirichlet boundary condition. 
\begin{theorem}\label{ect}
 Let $\Omega$ be a convexset with Lipschitz boundary, and $ v\in \mathcal{C}(\overline{\Omega})\cap\mathcal{C}^2(\Omega)$ is the solution for the following elliptic equation:
 \[\begin{cases}
  \Delta v+ h(v)=0,\quad&\text{in }\Omega,\\
  v=0,\quad&\text{on }\partial\Omega,
 \end{cases}\]
 where $h\in \mathcal{C}^2([0,\infty))$ is a non-increasing function. Then, $v$ is {\red{$p$}}-concave if $$g(t)=\begin{cases}
 t^{3-\frac{1}{p}}h(t^{\frac{1}{p}}),\quad &\text{if }p\neq0,\\
 e^{-t}h(e^t),\quad &\text{if }p=0,
 \end{cases}$$
 is concave in $[0,\infty)$ for $p\in[0,1]$.
\end{theorem}
The following lemma is from \cite[Property 2]{KA1985}

\begin{lemma}\label{cam} If $u$ is $\alpha$-concave for $\alpha>0$ in a convex domain, then $u$ is $\beta$-concave for all $0\leq \beta\leq \alpha$.
\end{lemma}


\vspace{4mm}

\noindent Q. He,\par\nopagebreak
\noindent\textsc{Laboratory of Computational and Quantitative Biology (LCQB), UMR 7238 CNRS, Sorbonne Universit\'e, 75205 Paris Cedex 06, France.}\newline
Email address: {\tt qingyou.he@sorbonne-universite.fr}

\end{document}